\documentclass[letterpaper, 10 pt, conference]{ieeeconf}  

\IEEEoverridecommandlockouts                              

\usepackage{amsmath}
\usepackage{color}
\usepackage{amssymb}  
\usepackage{amsfonts}
\usepackage{graphicx}
\usepackage{dsfont}
\usepackage{psfrag}
\usepackage{mathrsfs}
\usepackage{epstopdf}
\usepackage{mathtools}
\usepackage[noadjust]{cite}
\usepackage{bbm}
\usepackage{algorithm}
\usepackage{algorithmic}
\usepackage{savesym}
\usepackage{amsthm}

\theoremstyle{definition}
\newtheorem{definition}{Definition}

\theoremstyle{plain}
\newtheorem{theorem}{Theorem}

\newtheorem{lemma}{Lemma}
\newtheorem{probstat}{Problem}

\title{\LARGE \bf
A Separation Principle for Discrete-Time Fractional-Order \\ Dynamical Systems and its Implications to Closed-loop Neurotechnology}

\author{Sarthak Chatterjee$^{\dagger}$ \qquad Orlando Romero$^{\ddagger}$ \qquad S\'{e}rgio Pequito$^{\ddagger}$
\thanks{$^{\dagger}$Department of Electrical, Computer, and Systems Engineering, Rensselaer Polytechnic Institute, Troy, NY 12180, USA
        {\tt\small chatts3@rpi.edu}}%
\thanks{$^{\ddagger}$Department of Industrial and Systems Engineering, Rensselaer Polytechnic Institute, Troy, NY 12180, USA
        {\tt\small \{rodrio2,goncas\}@rpi.edu}}}%

\begin{document}

\maketitle
\thispagestyle{empty}
\pagestyle{empty}

\begin{abstract}

\mbox{Closed-loop} neurotechnology requires the capability to predict the state evolution and its regulation under (possibly) partial measurements. There is evidence that neurophysiological dynamics can be modeled by \mbox{fractional-order} dynamical systems. Therefore, we propose to establish a separation principle for \mbox{discrete-time} \mbox{fractional-order} dynamical systems, which are
inherently nonlinear and are able to capture spatiotemporal relations that exhibit \mbox{non-Markovian} properties. The separation principle states that the problems of controller and state estimator design can be done independently of each other while ensuring proper estimation and control in \mbox{closed-loop} setups. Lastly, we illustrate, as proof-of-concept, the application of the separation principle when designing controllers and estimators for these classes of systems in the context of neurophysiological data. In particular, we rely on real data to derive the models used to assess and regulate the evolution of \mbox{closed-loop} neurotechnologies based on electroencephalographic data.

\end{abstract}

\section{Introduction}

There is an increasing trend of looking into leveraging \emph{closed-loop} control and estimation strategies for the continued monitoring and interaction of subjects in the form of closed-loop neurotechnologies. Such technologies bring the promise of improving the quality-of-life of patients affected by neurological disorders such as epilepsy~\cite{rns}, Parkinson's disease~\cite{benabid2003deep}, Alzheimer's disease~\cite{nardone2015neurostimulation}, anxiety~\cite{sturm2007nucleus}, and depression~\cite{marangell2007neurostimulation}.  For neurophysiological signals, lingering interacting effects originating from long-term temporal dependence properties have illustrated the potential for clinical applications of \emph{fractional-order} based modeling, design, and analysis of such neurotechnologies~\cite{brodu2012exploring,ciuciu2012scale,zorick2013multifractal,zhang2015multifractal}.

Due to the highly dynamic nature of the neurophysiological processes, it is imperative that we consider \emph{feedback} mechanisms~\cite{andrzej}. A particularly successful closed-loop controller design strategy that has achieved remarkable success in several engineering applications is the strategy of \emph{model predictive control}~\cite{camacho2013model}.  Indeed, the main advent of \mbox{model-based} approaches is that we can understand how an external signal or stimulus would craft the dynamics of the process. In~\cite{romeroMPC}, the authors propose an electrical neurostimulation MPC-based strategy for the mitigation of epileptic seizures by modeling brain dynamics through \mbox{fractional-order} systems.


Recent work provides evidence that \mbox{fractional-order} dynamical systems (FODS) exhibit great success in accurately modeling dynamics which undergo nonexponential \mbox{power-law} decay, and have \mbox{long-term} memory or fractal properties~\cite{moon2008chaotic,lundstrom2008fractional,werner2010fractals,turcott1996fractal,thurner2003scaling,teich1997fractal}. Not only have FODS found applications in domains such as gas dynamics~\cite{chen2010anomalous}, viscoelasticity~\cite{jaishankar2012power}, chaotic systems~\cite{petravs2011fractional}, and biological swarming~\cite{west2014networks}, just to mention a few, but also in \mbox{cyber-physical} systems to model the interlaced evolution of the spatial and temporal components of complex networks~\cite{xuecps,xue2017reliable}. Some of these relationships have also been explored in the domain of neurophysiological signals such as electroencephalogram (EEG) and electrocardiogram (ECG)~\cite{magin2006fractional}.

The \emph{separation principle}, one of the cornerstones of modern feedback systems theory, states that the problems of optimal control and state estimation can be decoupled in certain specific instances~\cite{MitterStochControl}. These ideas were advanced early on in~\cite{JosephTouLinear,potter1964guidance} and~\cite{wonham1968separation} and are connected to the idea of \emph{certainty equivalence}~\cite{CertaintyEquiv} in stochastic control theory. Since then, the separation principle has been proposed in a wide variety of settings, including, but not limited to, stochastic control systems~\cite{Sep_Prin_CDC,Sep_Redux}, hybrid systems~\cite{bencze1995separation}, distributed control systems~\cite{rantzer2006separation}, quantum control~\cite{bouten2008separation}, linear systems with Markovian jumps~\cite{costa2003finite}, wireless fading channels subject to channel capacity constraints~\cite{charalambous2008control}, and \mbox{discrete-time} networked control systems with random packet drops~\cite{wu2011separation}.

However, the separation principle does not hold for nonlinear systems in general. Therefore, in this paper, we \textit{state and prove a separation principle result that stems from the problem of \mbox{closed-loop} \mbox{discrete-time} FODS and demonstrate the implications of our results in the context of \mbox{closed-loop} neurotechnology using \mbox{real-world} electroencephalographic data}. Specifically, we prove that if a \mbox{closed-loop} controller and an observer are designed for \mbox{discrete-time} FODS, then the aforementioned design can be carried out independently of each other. FODS are inherently nonlinear and they possess \mbox{long-term} memory in the sense that the evolution of a FODS aggregates the effects of \emph{all time} as the evolution of the system progresses. As a consequence, the innate \mbox{non-Markovian} nonlinearity of FODS does not immediately ensure the existence of a separation simple for the reasons mentioned above. Furthermore, FODS are finding increasing applications in the field of model predictive control (MPC), where the problems of estimator and controller design need the existence of a separation principle. Although separation principle results such as~\cite{echi_basdouri_benali_2019,naifar2018global,matignon1997observer} have been derived for FODS in continuous time, and, to the best of our knowledge, no such result has been previously proposed and analyzed for \mbox{discrete-time} FODS.

The remainder of the paper is organized as follows. Section~\ref{sec:fods} presents some essential theory regarding \mbox{discrete-time} FODS including the system model that we consider and the separation principle we propose to prove. Section~\ref{sec:proof} presents the proof of the separation principle for \mbox{discrete-time} FODS. Section~\ref{sec:examples} provides an illustrative example that shows how the separation principle can be used to sustain closed-loop feedback performance in the context of neurotechnology, and Section~\ref{sec:conclusion} concludes the paper.

\section{Problem Statement}\label{sec:fods}

We consider a deterministic linear \mbox{discrete-time} \mbox{fractional-order} dynamical system described as follows
\begin{align}\label{eq:frac_model}
    \Delta^\alpha x[k+1] &= Ax[k] + Bu[k] \nonumber \\
    y[k] &= Cx[k] \nonumber \\
    x[0] &= x_0,
\end{align}
where $x \in \mathbb{R}^n$ is the \emph{state} for time step $k \in \mathbb{N}$, $u \in \mathbb{R}^p$ is the \emph{input} and $y \in \mathbb{R}^n$ is the \emph{output}. $A \in \mathbb{R}^{n \times n}$ is the system matrix, $B \in \mathbb{R}^{n \times p}$ is the input matrix, and $C \in \mathbb{R}^{p \times n}$ is the sensor measurement matrix. Note that the system model is similar to a classic \mbox{discrete-time} linear \mbox{time-invariant} model but it is nonlinear due to the inclusion of the fractional derivative, whose expansion and discretization for the $i$-th state, $1 \leq i \leq n$, can be written as
\begin{equation}\label{eq:frac_deriv}
    \Delta ^{\alpha_i} x_i[k] = \sum_{j=0}^k \psi(\alpha_i,j) x_i[k-j],
\end{equation}
where $\alpha_i$ is the fractional order corresponding to state $i$ and
\begin{equation}
    \psi(\alpha_i,j) = \frac{\Gamma (j-\alpha_i)}{\Gamma (-\alpha_i) \Gamma (j+1)},
\end{equation}
with $\Gamma(\cdot)$ being the gamma function defined by $\Gamma (z) = \int_0^{\infty} s^{z-1} e^{-s} \: \mathrm{d}s$ for all complex numbers $z$ with $\Re (z) > 0$~\cite{DzielinskiFOS}.
\par Given the deterministic linear \mbox{discrete-time} \mbox{fractional-order} dynamical system~\eqref{eq:frac_model}, we have two main control objectives that need to be satisfied.
\begin{itemize}
    \item \textbf{Stabilizability:} In this problem, we deal with the issue of \emph{stabilization} of system~\eqref{eq:frac_model}. To this end, we consider the problem of designing a controller to stabilize the system~\eqref{eq:frac_model}.
The second control objective is concerned with designing an observer for~\eqref{eq:frac_model}.
    \item \textbf{Observer Design:} Assume that the states of~\eqref{eq:frac_model} are not known exactly. In this problem, we deal with the issue of designing an \emph{observer} for the system~\eqref{eq:frac_model}. The observer that we design should help us to estimate the states of the system given knowledge of the input $u \in \mathbb{R}^p$ and the output $y \in \mathbb{R}^n$.
\end{itemize}
With these two objectives in mind, we seek to prove the following result.
\begin{probstat}\label{prob:sep_prin}
Given the deterministic linear \mbox{discrete-time} \mbox{fractional-order} dynamical system~\eqref{eq:frac_model}, can the problems of stabilizability and observer design can be carried out independently of each other towards achieving closed-loop stabilizability with partial measurements?
\end{probstat}

\section{Separation Principle for \mbox{Fractional-Order} Systems}\label{sec:proof}

In this section, we will present the main result of our paper, i.e. the separation principle for \mbox{discrete-time} FODS. We first introduce the theory of state evolution in \mbox{discrete-time} FODS, presenting the relevant equations for the evolution of the dynamics of the system states in Lemma~\ref{lemma:state_prop}. We will then sequentially consider the problems of observer design (in Section~\ref{subsec:obsv}), which entails the construction of an observer for the dynamical system~\eqref{eq:frac_model}, followed by the problem of stabilizability (in Section~\ref{subsec:feedback}), which requires us to design a controller to stabilize the system~\eqref{eq:frac_model}. With the above ingredients, and some mathematical preliminaries, we present the statement and proof of the main result of our paper, the separation principle for \mbox{discrete-time} FODS in Section~\ref{subsec:sepa} (see Theorem~\ref{thm:sepprin}).

\par We begin by reviewing some essential theory for \mbox{fractional-order} systems, including \mbox{closed-form} expressions for the state dynamics. Using the expansion of the \mbox{fractional-order} derivative in~\eqref{eq:frac_deriv}, the evolution of the state vector can be written as follows

\begin{align}\label{eq:state_evol}
    x[k+1] &= Ax[k] - \sum_{j=1}^{k+1} D(\alpha,j) x[k+1-j] + Bu[k] \nonumber \\
    x[0] &= x_0,
\end{align}
where $D(\alpha,j) = \text{diag}(\psi(\alpha_1,j),\psi(\alpha_2,j),\ldots,\psi(\alpha_n,j))$. Alternatively,~\eqref{eq:state_evol} can be written as
\begin{align}\label{eq:state_evol_2}
    x[k+1] &= \sum_{j=0}^k A_j x[k-j] + Bu[k] \nonumber \\
    x[0] &= x_0,
\end{align}
where $A_0 = A - D(\alpha,1)$ and $A_j = -D(\alpha,j+1)$ for $j \geq 1$. Defining matrices $G_k$ as
\begin{equation}
    G_k =
    \begin{cases*}
      I_n & $k = 0$, \\
      \displaystyle \sum_{j=0}^{k-1} A_j G_{k-1-j} & $k \geq 1$,
    \end{cases*} 
\end{equation}
we can state the following result.
\begin{lemma}[\cite{guermah2008}]\label{lemma:state_prop}
The solution to the system described by~\eqref{eq:frac_model} is given by
\begin{equation}
    x[k] = G_k x[0] + \sum_{j=0}^{k-1} G_{k-1-j} Bu[j].
\end{equation}
\end{lemma}
Having obtained the \mbox{closed-form} expressions for the state vectors, we turn our attention to the problems of observer design and stabilizability.
\subsection{Observer Design}\label{subsec:obsv}

In this section, we will show that it is possible to obtain an unbiased estimator by considering an innovation term added to the dynamics of the state estimate, which can be described as follows. We consider the construction of a \mbox{Luenberger-like} observer~\cite{luenbergerbook}, whose state and output estimates are denoted by $\hat{x}[k] \in \mathbb{R}^n$ and $\hat{y}[k] \in \mathbb{R}^n$, respectively. This observer then takes the following form
\begin{align}\label{eq:obsv_dyn}
    \hat{x}[k+1] &= \sum_{j=0}^k A_j \hat{x}[k-j] + Bu[k] + L(y[k] - \hat{y}[k]), \nonumber \\
    \hat{y}[k] &= C \hat{x}[k],
\end{align}
where the matrix $L \in \mathbb{R}^{n \times n}$ is a weighting matrix that weights the difference between the outputs of the plant and the observer. Note that the observer consists of two parts, the first part being a copy of the plant's dynamics as applied to the observer, and an innovation term being a scaled version of the difference between the outputs of the plant and the observer.

\subsection{Stabilizability and Output Feedback}\label{subsec:feedback}

In this section, we consider the problem of stabilizing~\eqref{eq:state_evol_2} in a classical \mbox{state-feedback} control setting. Assume that the control input $u \in \mathbb{R}^p$ can be written a weighted linear combination of the states of the observer with memory, i.e.,
\begin{align}\label{eq:state_feedback}
u[k] &= F_0 \hat{x}[k] + F_1 \hat{x}[k-1] + \ldots + F_k \hat{x}[0] \nonumber \\
&= \sum_{j=0}^k F_j \hat{x}[k-j],
\end{align}
where $F_j \in \mathbb{R}^{p \times n}$ for $j=0,1,\ldots,k$. Substituting this into~\eqref{eq:obsv_dyn} and using the fact that $y[k] = Cx[k]$ and $\hat{y}[k] = C\hat{x}[k]$, we have
\begin{align}
    \hat{x}[k+1] &= \sum_{j=0}^k (A_j + BF_j) \hat{x}[k-j] + L(y[k] - \hat{y}[k]) \nonumber \\
    &= \sum_{j=0}^k (A_j + BF_j) \hat{x}[k-j] + L(Cx[k] - C\hat{x}[k]) \nonumber \\
    &= \sum_{j=0}^k (A_j + BF_j) \hat{x}[k-j] + LC e[k],
\end{align}
where $e[k] = x[k] - \hat{x}[k]$ is defined as the error between the states of the plant and the observer.
\par We now turn our attention towards the problem of output feedback. Going back to the dynamics of the plant and substituting~\eqref{eq:state_feedback} into~\eqref{eq:state_evol_2}, we have
\begin{align}\label{eq:plant_dyn}
    x[k+1] &= \sum_{j=0}^k A_j x[k-j] + Bu[k] \nonumber \\
    &= \sum_{j=0}^k A_j x[k-j] + \sum_{j=0}^k B F_j \hat{x}[k-j] \nonumber \\
    &= \sum_{j=0}^k A_j x[k-j] + \sum_{j=0}^k B F_j (x[k-j] - e[k-j]) \nonumber \\
    &= \sum_{j=0}^k (A_j + BF_j) x[k-j] - \sum_{j=0}^k BF_j e[k-j].
\end{align}
Next, we consider the dynamics of the error signal. Indeed, we have
\begin{align}\label{eq:error_dyn}
    e[k+1] &= x[k+1] - \hat{x}[k+1] \nonumber \\
    &= \sum_{j=0}^k (A_j + BF_j) x[k-j] - \sum_{j=0}^k BF_j e[k-j]  \nonumber \\ &- \left( \sum_{j=0}^k (A_j + BF_j) \hat{x}[k-j] + LC e[k] \right) \nonumber \\
    &= \sum_{j=0}^k (A_j + BF_j) e[k-j] - \sum_{j=0}^k BF_j e[k-j] \nonumber \\
    &- LC e[k] \nonumber \\
    &= \sum_{j=0}^k A_j e[k-j] - LC e[k].
\end{align}

\subsection{Separation Principle for \mbox{Discrete-Time} FODS}\label{subsec:sepa}

Having derived the expressions for the dynamics of the plant state and the error signal, we are now ready to state and prove the separation principle for \mbox{discrete-time} FODS. We first state some mathematical preliminaries that will aid our proof.
\begin{definition}\label{defn:hilbert}
A Hilbert space is a vector space $\mathscr{H}$ over $\mathbb{R}$ or $\mathbb{C}$ together with an inner product $\langle \cdot,\cdot \rangle$, such that relative to the metric $d(x,y) = \| x-y \|$ induced by the norm $\| \cdot \|^2 = \langle \cdot,\cdot \rangle$, $\mathscr{H}$ is a complete metric space.
\end{definition}
\begin{definition}\label{defn:seqsp}
The sequence space $\ell^2(\mathbb{N})$ denotes the Hilbert space of all \mbox{square-summable} sequences. Such sequences are represented by vectors with infinitely many elements $\mathcal{X} = \{ x[0],x[1],x[2],\ldots \}$. For $\mathcal{X}, \mathcal{Y} \in \ell^2(\mathbb{N})$, the space is equipped with the inner product
\[ \langle \mathcal{X},\mathcal{Y} \rangle = \sum_{k=0}^\infty x[k] y[k]^*, \]
where the $^*$ denotes the complex conjugate. In other words, a sequence $\mathcal{X} \in \ell^2(\mathbb{N})$ if $\| \mathcal{X} \|^2 = \langle \mathcal{X},\mathcal{X} \rangle = \sum_{k=0}^\infty | x[k] |^2 < \infty$.
\end{definition}
\begin{definition}\label{defn:backshift}
For a causal sequence $\mathcal{X}$, we define the \emph{backward shift operator} $\mathcal{S}: \ell^2(\mathbb{N}) \to \ell^2(\mathbb{N})$ by
\[ \mathcal{S} \mathcal{X} = \mathcal{S} \{ x[0],x[1],x[2],\ldots \} = \{ x[1],x[2],x[3],\ldots \}.\]
\end{definition}
\begin{definition}\label{defn:spec}
The \emph{spectrum} of a matrix $M$, denoted by $\mathsf{spec}(M)$ is the set of eigenvalues of the matrix $M$.
\end{definition}
Lastly, we present the main result of this paper.
\begin{theorem}[Separation Principle for \mbox{discrete-time} FODS]
\label{thm:sepprin}
Consider the \mbox{discrete-time} \mbox{fractional-order} dynamical system given in~\eqref{eq:frac_model}, and consider the problems of
\begin{enumerate}
    \item Designing an unbiased estimator (of the form~\eqref{eq:obsv_dyn}) for the system~\eqref{eq:frac_model} by following the procedure outlined in Section~\ref{subsec:obsv}, and,
    \item Designing a controller (of the form~\eqref{eq:state_feedback}) that stabilizes the system~\eqref{eq:frac_model} by following the procedure outlined in Section~\ref{subsec:feedback}.
\end{enumerate}
Then, given knowledge of the input $u \in \mathbb{R}^p$ and the output $y \in \mathbb{R}^n$, the above designs can be done independently of each other towards achieving closed-loop stabilizability with partial measurements.
\end{theorem}
\begin{proof}
With respect to our problem, we define the infinite column sequences
\begin{equation}
    \mathscr{X} = \begin{bmatrix} x[0] \\ x[1] \\ x[2] \\ \vdots \end{bmatrix}, \: \mathscr{E} = \begin{bmatrix} e[0] \\ e[1] \\ e[2] \\ \vdots \end{bmatrix}.
\end{equation}
Using $\mathscr{X}$ and $\mathscr{E}$, we can now compactly write equations~\eqref{eq:plant_dyn} and \eqref{eq:error_dyn} as follows
\begin{equation}\label{eq:compact}
    \begin{bmatrix} \mathcal{S}\mathscr{X} \\ \mathcal{S}\mathscr{E} \end{bmatrix} = \underbrace{\begin{bmatrix} \mathscr{J}_1 & \mathscr{J}_2 \\ \mathbf{0} & \mathscr{J}_3 \end{bmatrix}}_{\mathscr{J}} \begin{bmatrix} \mathscr{X} \\ \mathscr{E} \end{bmatrix},
\end{equation}
where $\mathcal{S}$ is the backward shift operator and the matrices $\mathscr{J}_i$ ($i=1,2,3$) are Toeplitz with the following structures
\begin{subequations}
\begin{equation}
    \mathscr{J}_1 = \begin{bmatrix} A_0 + BF_0 & 0 & 0 & \ldots & 0 \\
    A_1 + BF_1 & A_0+BF_0 & 0 & \ldots & 0 \\
    A_2 + BF_2 & A_1+BF_1 & A_0+BF_0 & \ldots & 0 \\
    \vdots & \vdots & \vdots & \ddots & \vdots
    \end{bmatrix},
\end{equation}
\begin{equation}
    \mathscr{J}_2 = \begin{bmatrix} -BF_0 & 0 & 0 & \ldots & 0 \\
    -BF_1 & -BF_0 & 0 & \ldots & 0 \\
    -BF_2 & -BF_1 & -BF_0 & \ldots & 0 \\
    \vdots & \vdots & \vdots & \ddots & \vdots
    \end{bmatrix},
\end{equation}
\begin{equation}
    \mathscr{J}_3 = \begin{bmatrix} A_0 - LC & 0 & 0 & \ldots & 0 \\
    A_1 - LC & A_0 - LC & 0 & \ldots & 0 \\
    A_2 - LC & A_1 - LC & A_0 - LC & \ldots & 0 \\
    \vdots & \vdots & \vdots & \ddots & \vdots
    \end{bmatrix}.
\end{equation}
\end{subequations}
Note that $\mathscr{J}_1$ only contains terms that pertain to the stabilizability, and $\mathscr{J}_3$ only contains terms that pertain to the observer design. From the block structure of~\eqref{eq:compact}, it can be seen that
\[ \mathsf{spec}(\mathscr{J}) = \mathsf{spec}(\mathscr{J}_1) \cup \mathsf{spec}(\mathscr{J}_3), \]
and the design of $\mathscr{J}_1$ and $\mathscr{J}_3$ can be carried out independently of each other.
\end{proof}

Although the design of the Luenberger-like observer in Section~\ref{subsec:obsv} starts with the design of a single weighting matrix $L$ that weights the outputs of the plant and the observer without memory, it is instructive to note that the separation principle for \mbox{discrete-time} FODS that we proved in Section~\ref{subsec:sepa} also holds for an observer of the following form
\begin{align}\label{eq:new_obsv_dyn}
    \hat{x}[k+1] &= \sum_{j=0}^k A_j \hat{x}[k-j] + Bu[k] \nonumber \\ &+ \sum_{j=0}^k L_j(y[k-j] - \hat{y}[k-j]), \nonumber \\
    \hat{y}[k] &= C \hat{x}[k].
\end{align}
The key difference between the observers in equations~\eqref{eq:obsv_dyn} and~\eqref{eq:new_obsv_dyn} are that in the former we have a single weighting matrix that weights the difference of the outputs of the plant and the observer, and in the latter, we use multiple weighting matrices to weight the differences of the outputs of the plant and the observer with memory. We then have the following theorem.
\begin{theorem}
\label{thm:new_sepprin}
Consider the \mbox{discrete-time} \mbox{fractional-order} dynamical system given in~\eqref{eq:frac_model}, and consider the problems of
\begin{enumerate}
    \item Designing an unbiased estimator (of the form~\eqref{eq:new_obsv_dyn}) for the system~\eqref{eq:frac_model} by following the procedure outlined above, and,
    \item Designing a controller (of the form~\eqref{eq:state_feedback}) that stabilizes the system~\eqref{eq:frac_model} by following the procedure outlined in Section~\ref{subsec:feedback}.
\end{enumerate}
Then, given knowledge of the input $u \in \mathbb{R}^p$ and the output $y \in \mathbb{R}^n$, the above designs can be done independently of each other towards achieving closed-loop stabilizability with partial measurements.
\end{theorem}

\begin{proof}
By setting $L_0=L$, and $L_1 = L_2 = \ldots = L_k = 0$ for $k=1,2,\ldots$, the problem reduces to the statement of Theorem~\ref{thm:sepprin}, and the proof follows by a similar line of reasoning.
\end{proof}

\section{Closed-Loop Neurotechnology}\label{sec:examples}
In this section, we illustrate our results by designing a \emph{model predictive controller} (MPC) that simulates a simple implantable closed-loop electrical neurostimulator. The controller is implemented on a \mbox{discrete-time} fractional-order plant, representing normal brain activity, whereas the predictive model will be based on an autoregressive finite-history approximation. Naturally, the controller will be designed as if it had access to the actual state of the system, and similarly the state observer (whose estimates are fed into the designed controller) is designed without consideration of the control strategy adopted.

We start by identifying the spatial and temporal parameters $A$ and $\alpha$ in~\eqref{eq:frac_model}, from a 4-channel sample of length 1 second of normalized EEG recordings. We model the $n=4$ components of the state vector as denoting the different recorded channels (\emph{i.e.} readings obtained from microelectrodes). The data used for these experiments are from subject 11 from the CHB-MIT Scalp EEG database~\cite{PhysioNet}. To achieve this identification, we leveraged the tools developed in~\cite{gaurav2017acc}, which led us to
\begin{equation}
    A = \left[\begin{array}{cccc} 0.0350 & 0.0526 & -0.0034 & -0.0391\\ 0.0296 & -0.0496 & 0.0646 & 0.0610\\ -0.0103 & -0.0028 & -0.0091 & 0.0068\\ -0.0291 & 0.0143 & -0.0008 & 0.0394 \end{array}\right]
\end{equation}
and
\begin{equation}
    \alpha = \left[\begin{array}{cccc} 0.5945 & 0.7176 & 0.9603 & 0.6279 \end{array}\right]^\mathsf{T},
\end{equation}
as the main parameters in the system. We are interested in modeling the impact of an electrical stimulation signal $u[k]$ originating from an integrated arbitrary voltage generator circuit. We start by considering the scenario \mbox{$B = \begin{bmatrix}1& 1& 1 & 1\end{bmatrix}^\mathsf{T}$} corresponding to a stimulus that perturbs all channels uniformly (\emph{e.g.}, if the four electrodes are placed considerably near each other). The measurements $y[k]$ used to estimate the state (through a simple \mbox{Kalman-like} filter) will be assumed simply as those given directly by the first channel, \emph{i.e.}, $C = \begin{bmatrix}1 & 0 & 0 & 0\end{bmatrix}$.

At each step $k$, the MPC controller will minimize a quadratic cost function
\begin{equation}
    J(u[k],\ldots,u[k+P-1]) = \sum_{j=1}^P\|x[k] - x_\mathrm{ref}[k+j]\|^2,
\end{equation}
with the predicted evolution $x[k]$ evolving not by the original system~\eqref{eq:frac_model}, by instead by a multivariate autoregressive (MVAR) approximation
\begin{equation}
    x[k+1] = \sum_{j=0}^{p-1}A_j x[k-j] + Bu[k],
\end{equation}
based on~\eqref{eq:state_evol_2}, by clipping off the infinite-horizon memory dependence by instead only a $p$-horizon one. The \emph{prediction horizon} $P$ was set to $P=8$ (50 milliseconds), whereas the \emph{control horizon} $M$ upon which the solution is implemented was set to $M=4$ (25 milliseconds). The reference signal~$x_\mathrm{ref}[k]$ denotes a simple rectangular pulse of frequency \mbox{8 Hz}, within the usual range of \emph{alpha rhythms} that characterize relaxed, but conscious brain activity~\cite{alpha}.

\begin{figure}[!ht]
    \centering
    \includegraphics[scale = 0.55]{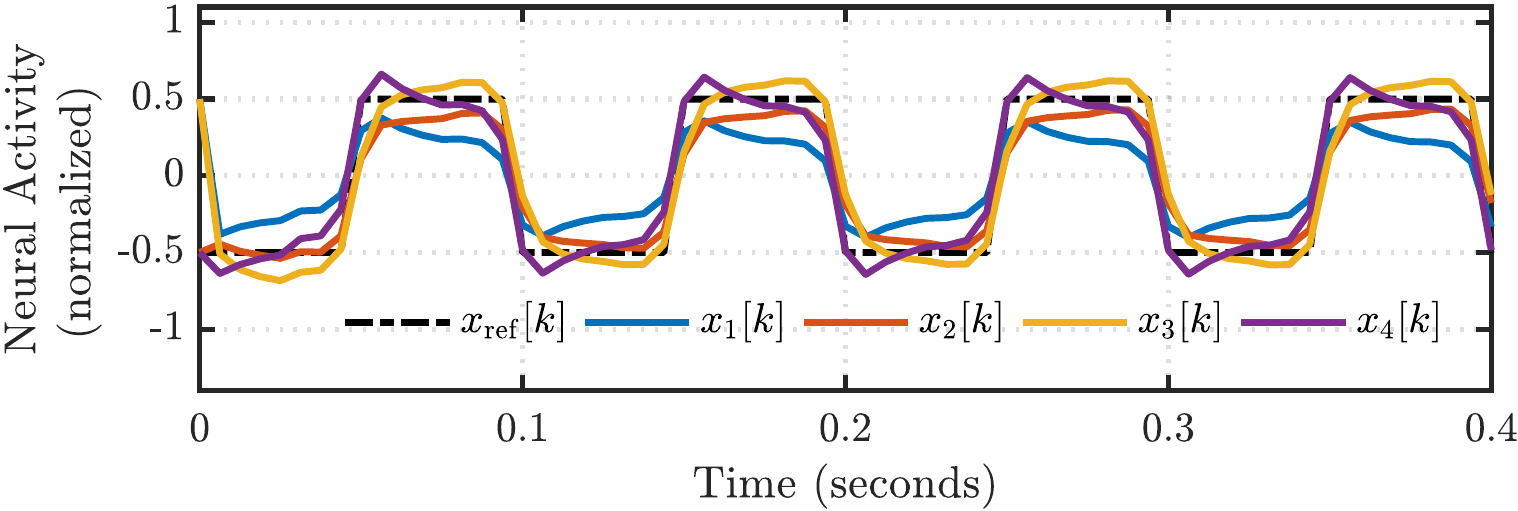}
    \caption{MPC-based neuromodulation of a \mbox{discrete-time} FODS representing normal brain activity by state tracking of a rectangular pulse of frequency 8 Hz, roughly in the range of alpha rhythms.}
    \label{fig:MPC}
\end{figure}

The results can be seen in Fig.~\ref{fig:MPC}, and as we can see, the controller is largely successful despite never having direct access to the state of the system. In other words, efficient design of a closed-loop controller and observer can be carried out separately for discrete-time fractional-order systems, as formally established in this paper.


\section{Conclusions and Future Work}\label{sec:conclusion}

In this paper, we proposed and proved a separation principle result for \mbox{discrete-time} FODS. As a consequence, we can decouple the problems of designing, respectively, a controller for the stabilization of the system states, and an observer for the estimation of the system states. The ability of \mbox{discrete-time} FODS to model complex spatiotemporal relationships in neurophysiological signals have led to the use of these models in closed-loop neurotechnologies.



Very rarely in practical settings, however, do we have deterministic \mbox{fractional-order} models. EEG signals, for instance, are particularly prone to disturbances arising from outside the brain, which are referred to as \emph{artifacts} in the neuroscience literature~\cite{britton2016electroencephalography}. Furthermore, stabilizing these models in the presence of disturbances becomes relevant in the treatment of disorders like epilepsy, Parkinson's disease, or Alzheimer's disease, since, in recent years, there have been increasing research efforts into finding possible palliative therapies for the aforementioned using neurofeedback~\cite{marzbani2016neurofeedback}. Future work, therefore, will focus on developing controllers and observers for FODS with associated process and measurement noise, and investigating the possible existence of separation \mbox{principle-like} results akin to those already existing in the field of linear stochastic control theory.






\bibliographystyle{IEEEtran}
\bibliography{IEEEabrv,mybibfile}

\end{document}